\documentclass[11pt,oneside,english]{amsart}
\usepackage{lmodern}
\usepackage[T1]{fontenc}
\usepackage[latin9]{inputenc}
\usepackage{geometry}
\usepackage{amsthm}

\makeatletter
\numberwithin{equation}{section}
\numberwithin{figure}{section}
\theoremstyle{plain}
\newtheorem{thm}{\protect\theoremname}[section]
\theoremstyle{plain}
\newtheorem{lem}[thm]{\protect\lemmaname}
\ifx\proof\undefined
\newenvironment{proof}[1][\protect\proofname]{\par
\normalfont\topsep6\p@\@plus6\p@\relax
\trivlist
\itemindent\parindent
\item[\hskip\labelsep\scshape #1]\ignorespaces
}{%
\endtrivlist\@endpefalse
}
\providecommand{\proofname}{Proof}
\fi

\date{}

\address{
Department of Pure Mathematics\\
University of Waterloo\\
Waterloo, Ontario, N2L 3G1, Canada}
\email{ztanko@uwaterloo.ca}


\AtBeginDocument{
  
}

\makeatother

\usepackage{babel}
\providecommand{\lemmaname}{Lemma}
\providecommand{\theoremname}{Theorem}

\begin{document}

\title{cyclicity of the left regular representation of a locally compact
group}

\author{Zsolt Tanko}
\begin{abstract}
We combine harmonic analysis and operator algebraic techniques to
give a concise argument that the left regular representation of a
locally compact group is cyclic if and only if the group is first
countable, a result first proved by Greenleaf and Moskowitz.
\end{abstract}

\maketitle
Let $G$ be a locally compact group and let $\lambda$ and $\rho$
denote the (unitarily equivalent) \emph{left and right regular representations}
of $G$ on $L^{2}\left(G\right)$, respectively. The \emph{group von
Neumann algebra} $VN\left(G\right)$ is the von Neumann algebra generated
in $B\left(L^{2}\left(G\right)\right)$ by $\lambda\left(G\right)$.
It is well known that the commutant of $VN\left(G\right)$ is the
von Neumann algebra generated by $\rho\left(G\right)$. In \cite{GM1},
an operator algebraic argument viewing $VN\left(G\right)$ as arising
from a left Hilbert algebra, in combination with a reduction argument
using the structure theory of locally compact groups, is used to show
that $\lambda$ is cyclic when the group $G$ is first countable.
The converse, left open in \cite{GM1}, was established later by the
same authors in \cite{GM2}. The purpose of this note is to give a
new and more economical proof of this equivalence. Moreover, we show
that these conditions are equivalent to $\sigma$-finiteness of $VN\left(G\right)$,
the latter condition arising naturally from our techniques. In the
commutative case, it is well known that $\sigma$-finiteness of $L^{\infty}\left(G\right)$
characterizes $\sigma$-compactness of $G$, and it is our hope that
further development of the techniques we employ will yield natural
characterizations of $\sigma$-finiteness of a general locally compact
quantum group. An alternative proof of the characterization we establish,
exploiting the structure theory of locally compact groups, is given
in \cite{LR}. 

Recall that the \emph{support} of a normal state $\omega$ on a von
Neumann algebra $M$ is the minimal projection $S_{\omega}$ in $M$
for which $\left\langle \omega,S_{\omega}\right\rangle =1$ and that
$\omega$ is \emph{faithful} if $S_{\omega}=I$, the identity in $M$,
equivalently if $\omega$ takes strictly positive values on strictly
positive operators. We record some elementary facts about these concepts.
For a vector $\xi$ in a Hilbert space, let $\omega_{\xi}$ denote
the vector functional implemented by $\xi$. The notation $\left\langle X\right\rangle $
denotes the norm closed linear span of $X$.
\begin{lem}
\label{lem:Normal_state_facts}Let $\mathcal{H}$ be a Hilbert space,
let $M$ a von Neumann algebra in $B\left(\mathcal{H}\right)$, and
let $\xi,\eta\in\mathcal{H}$ be unit vectors. The following hold.

\begin{enumerate}
\item The projection $S_{\omega_{\xi}}$ has range $\left\langle M^{\prime}\xi\right\rangle $.
\item $\left\langle \omega_{\xi},S_{\omega_{\eta}}\right\rangle =0$ if
and only if $\xi$ is orthogonal to $\left\langle M^{\prime}\eta\right\rangle $.
\item A projection $P$ in $M$ satisfies $P\xi=\xi$ if and only if $S_{\omega_{\xi}}\leq P$.
\item A normal state $\omega$ on $M$ is faithful if and only if $\left\langle \omega,U\right\rangle =\left\langle \omega,I\right\rangle $
implies $U=I$, for any unitary $U$ in $M$.
\end{enumerate}
\end{lem}
Motivated by the following simple observation, we choose to characterize
cyclicity of the right regular representation.
\begin{lem}
\label{lem:Cyclic_iff_faithful}Let $G$ be a locally compact group.
A vector $\xi\in L^{2}\left(G\right)$ is cyclic for $\rho$ if and
only if $\omega_{\xi}$ is faithful on $VN\left(G\right)$.
\end{lem}
\begin{proof}
For $\xi\in L^{2}\left(G\right)$ we have $\left\langle \rho\left(G\right)\xi\right\rangle =\left\langle VN\left(G\right)^{\prime}\xi\right\rangle $
since $\mbox{span}\rho\left(G\right)$ is strong operator topology
dense in $VN\left(G\right)^{\prime}$. Consequently, the vector $\xi$
is cyclic for $\rho$ if and only if $\left\langle VN\left(G\right)^{\prime}\xi\right\rangle =L^{2}\left(G\right)$.
As $S_{\omega_{\xi}}$ has range $\left\langle VN\left(G\right)^{\prime}\xi\right\rangle $,
the latter assertion is exactly that $S_{\omega_{\xi}}=I$.
\end{proof}
Let $A\left(G\right)$ denote the \emph{Fourier algebra} of a locally
compact group $G$, which is the predual of $VN\left(G\right)$, and
for $T\in VN\left(G\right)$ and $u\in A\left(G\right)$ define $T\check{\cdot}u\in A\left(G\right)$
by 
\begin{eqnarray*}
\left\langle T\check{\cdot}u,S\right\rangle =\left\langle u,\check{T}S\right\rangle  &  & \left(S\in VN\left(G\right)\right),
\end{eqnarray*}
where $\check{T}$ is the image of $T$ under the adjoint of the check
map $u\mapsto\check{u}$ on $A\left(G\right)$ (here, $\check{u}\left(s\right)=u\left(s^{-1}\right)$).
See \cite[p.213]{Eymard}. Proposition 3.17 of \cite{Eymard} shows
that for $u\in A\left(G\right)\cap L^{2}\left(G\right)$ we have $T\check{\cdot}u=Tu$,
where the right hand side is evaluation of the operator $T$ at the
vector $u$ in $L^{2}\left(G\right)$. This fact is needed in the
following lemma, which is key to establishing the main result.
\begin{lem}
\label{lem:Proj_rng_cts_fun}Let $G$ be a locally compact group.
Every nonzero projection in $VN\left(G\right)$ has a nonzero continuous
function in its range.
\end{lem}
\begin{proof}
Let $P\in VN\left(G\right)$ be a nonzero projection and choose a
unit vector $\xi$ in its range. Since positive functions span $\mathcal{C}_{c}\left(G\right)$,
which is in turn dense in $L^{2}\left(G\right)$, we may find a positive
$f\in\mathcal{C}_{c}\left(G\right)$ of norm one in $L^{2}\left(G\right)$
and not orthogonal to $\left\langle \rho\left(G\right)\xi\right\rangle $,
so that $\left\langle \omega_{f},S_{\omega_{\xi}}\right\rangle \neq0$
by Lemma \ref{lem:Normal_state_facts}. The function $\omega_{f}$
in $A\left(G\right)$ is positive definite and pointwise positive,
so that $\check{\omega_{f}}=\omega_{f}$, and is in $A\left(G\right)\cap L^{2}\left(G\right)$
because $f$ has compact support, whence 
\[
S_{\omega_{\xi}}\left(\omega_{f}\right)\left(e\right)=\left(S_{\omega_{\xi}}\check{\cdot}\omega_{f}\right)\left(e\right)=\left\langle S_{\omega_{\xi}}\check{\cdot}\omega_{f},\lambda\left(e\right)\right\rangle =\left\langle \omega_{f},\check{S_{\omega_{\xi}}}\right\rangle =\left\langle \check{\omega_{f}},S_{\omega_{\xi}}\right\rangle =\left\langle \omega_{f},S_{\omega_{\xi}}\right\rangle \neq0.
\]
Thus $S_{\omega_{\xi}}\left(\omega_{f}\right)=S_{\omega_{\eta}}\check{\cdot}\omega_{f}$
is nonzero and in $A\left(G\right)$, hence continuous, and is in
the range of $P$ because $S_{\omega_{\xi}}\leq P$, by Lemma \ref{lem:Normal_state_facts}. 
\end{proof}
\begin{thm}
Let $G$ be a locally compact group. The following are equivalent:

\begin{enumerate}
\item $G$ is first countable.
\item VN$\left(G\right)$ is $\sigma$-finite.
\item The left (equivalently, right) regular representation is cyclic.
\end{enumerate}
\end{thm}
\begin{proof}
Suppose (1) holds. Let $\left(U_{n}\right)_{n=1}^{\infty}$ be a countable
neighborhood base at the identity in $G$ and define $\omega_{n}=\left|U_{n}\right|^{-1}\omega_{\chi_{U_{n}}}$.
We show that the normal state $\omega=\sum_{n=1}^{\infty}2^{-n}\omega_{n}$
is faithful. Let $T$ be a positive operator in $VN\left(G\right)$
with $\left\langle \omega,T\right\rangle =0$ and let $P$ be the
range projection of $T$, so that $\left\langle \omega,P\right\rangle =0$
(see, e.g., \cite[Remark 7.2.5]{KRII}). Given any vector $\eta$
in the range of $T$ we have $S_{\omega_{\eta}}\leq P$ and thus $0\leq\left\langle \omega_{n},S_{\omega_{\eta}}\right\rangle \leq\left\langle \omega_{n},P\right\rangle \leq\left\langle \omega,P\right\rangle =0$,
implying that $\eta$ is orthogonal to $\left\langle \rho\left(G\right)\chi_{U_{n}}\right\rangle $
for each $n\geq1$. If $\eta$ is continuous, then
\[
\eta\left(s\right)=\lim_{n}\left|U_{n}s\right|^{-1}\int_{U_{n}s}\eta=\lim_{n}\left|U_{n}s\right|^{-1}\left\langle \eta|\rho\left(s^{-1}\right)\chi_{U_{n}}\right\rangle \Delta\left(s\right)^{\frac{1}{2}}=0
\]
for every $s\in G$. Thus $P=0$ by Lemma \ref{lem:Proj_rng_cts_fun},
hence $T=0$ and $\omega$ is faithful, so (2) holds.

Normal states on $VN\left(G\right)$, being positive definite functions
in $A\left(G\right)$, are vector states, so that statements (2) and
(3) are equivalent by Lemma \ref{lem:Cyclic_iff_faithful}.

We provide the argument of \cite{NR} establishing that (2) implies
(1). Suppose (2) holds and let $\omega$ be a faithful normal state
on $VN\left(G\right)$. Fix a compact neighborhood $K$ of the identity
in $G$ and let $V$ be any open neighborhood of the identity contained
in $K$. We show that the sets $U_{n}=\left\{ s\in K:\left|\omega\left(s\right)-1\right|<\frac{1}{n}\right\} $
form a neighborhood base at the identity, for which it suffices to
establish that $U_{n}$ is contained in $V$ for some $n\geq1$. For
any $s\in G$ with $\omega\left(s\right)=1$, Lemma \ref{lem:Normal_state_facts}
entails that $s=e$, since $\omega\left(s\right)=\left\langle \omega,\lambda\left(s\right)\right\rangle $.
Compactness of $K\setminus V$ then implies that $\epsilon=\inf\left\{ \left|\omega\left(s\right)-1\right|:s\in K\setminus V\right\} $
is strictly positive. Choosing $N\geq1$ with $\frac{1}{N}<\epsilon$,
if $s\in U_{N}$, then that $s\in K$ and $\left|\omega\left(s\right)-1\right|<\epsilon$
together imply that $s\in V$. Thus $U_{N}\subset V$, as required.
\end{proof}

\end{document}